%% file: Master-character-varieties.tex
\documentclass[a4paper,12pt]{amsart}
\usepackage{amsmath, amscd, amssymb, mathrsfs, mathtools, fullpage, enumerate, svninfo,color,enumitem}
\usepackage[all]{xy}
\usepackage{yfonts}
\usepackage{stmaryrd}
\usepackage{hyperref}

\title{Another look at $p$-adic  Fourier-theory}

\author{Guido Kings}
\address{Fakult\"at f\"ur Mathematik \\
Universit\"at Regensburg\\
93040 Regensburg\\
Germany}
\author{Johannes Sprang}
\address{
Fakultät für Mathematik\\
Universität Duisburg-Essen\\
45127 Essen\\
Germany}

\thanks{This research was supported by the DFG grant: SFB 1085 Higher invariants}

\theoremstyle{plain}
    \newtheorem{theorem}{Theorem}[section]
\newtheorem*{theorem*}{Theorem}
    \newtheorem{lemma}[theorem]{Lemma}
    \newtheorem{proposition}[theorem]{Proposition}
    \newtheorem{corollary}[theorem]{Corollary}
    \newtheorem*{corollary*}{Corollary}

\theoremstyle{remark}
    \newtheorem{remark}[theorem]{Remark}
    
     \newtheorem{example}[theorem]{Example}

\theoremstyle{definition}
    \newtheorem{definition}[theorem]{Definition}

\DeclareMathOperator{\Sym}{Sym}
\DeclareMathOperator{\Hom}{Hom}

\DeclareMathOperator{\End}{End}

\DeclareMathOperator{\Spf}{Spf}

\DeclareMathOperator{\Sp}{Sp}

\DeclareMathOperator{\Gal}{Gal}

\DeclareMathOperator{\Lie}{Lie}

\newcommand{\sG}{\mathscr{G}}

\newcommand{\sO}{\mathscr{O}}

\newcommand{\cI}{\mathcal{I}}

\newcommand{\cO}{\mathcal{O}}

\newcommand{\frt}{\mathfrak{t}}

\newcommand{\bbA}{\mathbb{A}}
\newcommand{\CC}{\mathbb{C}}

\newcommand{\QQ}{\mathbb{Q}}

\newcommand{\ZZ}{\mathbb{Z}}

\newcommand{\NN}{\mathbb{N}}

\newcommand{\Q}{\mathbb{Q}}

\newcommand{\C}{\mathbb{C}}
\newcommand{\Qp}{{\QQ_p}}
\newcommand{\Zp}{{\ZZ_p}}

\newcommand{\pDiv}{\mathbf{pDiv}}
\newcommand{\HTpairs}{\mathbf{HTpairs}}
\newcommand{\HTDpairs}{\mathbf{HT}^\vee\mathbf{pairs}}
\newcommand{\Adic}{\mathbf{Adic}}
\newcommand{\Rigid}{\mathbf{Rigid}}
\newcommand{\RigidGr}{\mathbf{RigidGr}}
\newcommand{\CatCharVar}{\mathbf{CharVar}}

\newcommand{\grT}{T}
\newcommand{\CharVar}[1]{\widehat{#1}}

\newcommand{\mult}{\mathrm{mult}}

\newcommand{\rig}{\mathrm{rig}}
\newcommand{\ad}{\mathrm{ad}}
\newcommand{\an}{\mathrm{an}}
\newcommand{\Lan}{\text{L-an}}

\newcommand{\isom}{\cong}

\newcommand{\id}{\mathrm{id}}

\def\endpiece{xxx}%                                                                     marks end of list
\def\makeAlphabet[#1]{\expandafter\makeA#1,xxx,}%               Ex. \makeAlphabet[A,B]
\def\makealphabet[#1]{\expandafter\makea#1,xxx,}%               Ex. \makealphabet[c,d]
\def\makeA#1,{\def\temp{#1}\ifx\temp\endpiece\else%
\mkbb{#1}\mkfrak{#1}\mkbf{#1}\mkcal{#1}\mkscr{#1}\expandafter\makeA\fi}%
\def\makea#1,{\def\temp{#1}\ifx\temp\endpiece\else\mkfrak{#1}\mkbf{#1}\expandafter\makea\fi}%
\def\mkbb#1{\expandafter\def\csname bb#1\endcsname{\mathbb{#1}}}%   
\def\mkfrak#1{\expandafter\def\csname fr#1\endcsname{\mathfrak{#1}}}%    Define frak
\def\mkbf#1{\expandafter\def\csname b#1\endcsname{\mathbf{#1}}}%           Define bold letters
\def\mkcal#1{\expandafter\def\csname c#1\endcsname{\mathcal{#1}}}%       Define calligraphy
\def\mkscr#1{\expandafter\def\csname s#1\endcsname{\mathscr{#1}}}%       Define script
\def\makeop[#1]{\xmakeop#1,xxx,}%                                       Ex. \makeop[Hom,Spec]
\def\mkop#1{\expandafter\def\csname #1\endcsname{{\mathrm{#1}}}} % 
\def\xmakeop#1,{\def\temp{#1}\ifx\temp\endpiece\else\mkop{#1}\expandafter\xmakeop\fi}%

\makeop[Mic,Mod,im]
\newcommand{\Gmf}{\widehat{\mathbb{G}}_m}
\newcommand{\OCp}{\sO_{\CC_p}}
\newcommand{\Cp}{{\CC_p}}

\AtBeginDocument{%
   \def\MR#1{}
}

\numberwithin{equation}{section}
\begin{document}

\begin{abstract}
In this short note, we show that a natural generalization of the $p$-adic Fourier theory of Schneider and Teitelbaum follows immediately from the classification of $p$-divisible groups over $\OCp$ by Scholze and Weinstein.
\end{abstract}

\setcounter{tocdepth}{1}
\maketitle

%\tableofcontents
\input{introduction}
\input{character-varieties}

\renewcommand{\MR}[1]{\relax\ifhmode\unskip\space\fi MR \MRhref{#1}{#1}}
\renewcommand{\MRhref}[2]{%
  \href{http://www.ams.org/mathscinet-getitem?mr=#1}{#2}.
}
\providecommand{\bysame}{\leavevmode\hbox to3em{\hrulefill}\thinspace}

\bibliographystyle{amsalpha}
\bibliography{character-varieties.bib}
\end{document}

%% file: introduction.tex
\section*{Introduction}
In \cite{Amice64} and \cite{Amice78}, Amice has given an explicit description of the space of $K$-valued locally $\Qp$-analytic distributions $D(\Zp,K)$ on $\Zp$, when $K$ is a complete subfield of $\Cp$. 

In an influential paper \cite{ST}, Schneider and Teitelbaum have generalized the work of Amice to the $K$-Fr\'echet algebra $D^{\Lan}(\sO_L,K)$ of locally $L$-analytic $K$-valued distributions on $\sO_L$, where $\Q_p\subset L\subset K\subset \C_p$ are field extensions and $\sO_L$ is the ring of integers in $L$. To describe these distributions, they introduce a character variety, which is a rigid analytic group variety $\widehat{\sO_L}^{\Lan}$ whose closed points in a field $K$ parametrize the $K$-valued locally $L$-analytic characters of $\sO_L$. In this setting their Fourier-transform is an isomorphism
\begin{equation*}
    F\colon D^{\Lan}(\sO_L,K)\isom \cO(\widehat{\sO_L}^{\Lan}/K),
\end{equation*}
where $\cO(\widehat{\sO_L}^{\Lan}/K)$ is the ring of global sections of the rigid variety $\widehat{\sO_L}^{\Lan}/K$. Their main theorem gives furthermore a description of the character variety over $\C_p$ in terms of the rigid generic fiber of a Lubin--Tate formal group with endomorphisms by $\sO_L$. The $p$-adic Fourier theory of Schneider and Teitelbaum has plenty of applications in number theory and representation theory. On the one hand, the $p$-adic Fourier theory has been studied extensively with the goal of generalizing the $p$-adic local Langlands correspondence from $\mathrm{GL}_2(\mathbb{Q}_p)$ to $\mathrm{GL}_2(L)$. On the other hand, the theory of $p$-adic distributions plays a key role in the construction of $p$-adic $L$-functions. For example, Schneider and Teitelbaum have constructed a $p$-adic $L$-function for Hecke characters of imaginary quadratic fields at inert primes $p$ as an application of their theory. 

In this note, we introduce what we call \emph{character groups with differential conditions}. It turns out that these character groups  carry in a natural way the structure of a rigid analytic variety. The character varieties considered by Schneider--Teitelbaum and Amice are a special case of our character varieties with differential conditions, as the $L$-analyticity can be expressed in a natural way as a  Cauchy--Riemann type differential condition. We show that the functions of our character varieties  can be interpreted as certain Fr\'echet spaces of distributions, in exactly the same fashion Schneider--Teitelbaum. We show with the theory of Scholze--Weinstein  that over $\C_p$ our character varieties are uniformized by $p$-divisible groups. In fact, each $p$-divisible group appears as a character variety with such a differential condition. 

In a follow up paper we will apply our results to the construction of $p$-adic $L$-functions for arbitrary critical Hecke characters for primes $p$ of  arbitrary split behaviour using the results in \cite{Kings-Sprang}. 
It is in this setting that the (higher dimensional) character varieties with differential conditions naturally appear and where the restriction to just $L$-analytic functions is no longer sufficient.

After this work was completed, we were informed by van Hoften, Howe and Graham, that they obtained similar results independently. 

%% file: character-varieties.tex
\section{Character groups with differential conditions}
Let $ K$ be a field extension of $\Q_p$ which is complete with respect to a non-archimedean valuation $|\_ |$, extending the valuation on $\Qp$. Let $\grT$ be a free $\Zp$-module of finite rank. We will view $\grT$ as a $\Qp$-analytic group and denote by 
\begin{equation*}
    C^{\an}(\grT,K)
\end{equation*}
the space of $K$-valued locally $\Qp$-analytic functions on $\grT$. 
\begin{definition}
Let
\begin{equation*}
    \CharVar{\grT}(K):=\{\chi\in \Hom(\grT,K^{\times})\mid \chi\in C^{\an}(\grT,K)\}
\end{equation*}
be the group of locally $\Qp$-analytic $K$-valued characters on $\grT$. 
\end{definition}
The Lie algebra $\frt$ of $\grT$ is simply $\grT\otimes_{\Zp}\Qp$. Let us write $\frt^\vee:=\Hom_{\Qp}(\frt,\Qp)$ for the co-Lie algebra. We get the differential
\[
	d\colon C^{\an}(\grT,K)\to C^{\an}(\grT,K)\otimes_{\Qp} \frt^\vee.
\]
We consider subgroups of $C^{\an}(\grT,K)$ and $\CharVar{\grT}(K)$ defined by differential conditions.
\begin{definition}
Let $L$ be a complete field extension of $\Qp$ and
\begin{equation*}
    W\subset  \Hom_{\Zp}(\grT,L)=\Hom_{\Qp}(\frt,L)
\end{equation*}
be a $L$-subvector space. For any complete field $K$ with $L\subseteq K$, we define
\begin{equation*}
    C^{\an}_W(\grT,K):=\left\{f \in C^{\an}(\grT,K)\mid d\chi\in  C^{\an}(\grT,K)\otimes_L W\right\},
\end{equation*}
and
\begin{equation*}
    \CharVar{\grT}_W(K):=\CharVar{\grT}(K) \cap C_W^{\an}(\grT,K).
\end{equation*}
We call $C_W^{\an}(\grT,K)$ the space of locally $\Qp$-analytic functions with differential condition $W$, and $\CharVar{\grT}_W(K)$ the \emph{character group with differential condition $W$}. 
\end{definition}

\begin{remark} Of course, one has $C_W^{\an}(T,K)=C_{W_K}^{\an}(T,K)$ for $W_K:=W\otimes_L K$. The distinction between $L$ and $K$ will be important later, as we will see that $\CharVar{\grT}_W(K)$ are the closed points in $K$ of a rigid analytic variety defined over $L$. 
\end{remark}

Let us consider some examples.
\begin{example}\label{exmpl:DistrQp}
	In the case $W=\Hom_{\Zp}(T,L)$, the differential condition $W$ does not impose any restrictions and we obtain
	\begin{align*}
		C_{\Hom_{\Zp}(T,L)}^{\an}(\grT,K)&= C^{\an}(T,K),\\
		\CharVar{\grT}_{\Hom_{\Zp}(T,L)}(K)&= \CharVar{\grT}(K).
	\end{align*}
\end{example}
The other extreme $W=0$ gives:
\begin{example}\label{exmpl:DistrTriv}
	In the case $W=\{0\}$, the only locally $L$-analytic functions with trivial differential are the locally constant functions $C^{\text{lc}}(\grT,K)$, and we obtain:
	\begin{align*}
		C_{\{0\}}^{\an}(\grT,K)&= C^{\text{lc}}(T,K),\\
		\CharVar{\grT}_{\{0\}}(K)&= \CharVar{\grT}^{\text{lc}}(K),
	\end{align*}
	where $\CharVar{\grT}^{\text{lc}}(K)$ denotes the locally constant (and hence finite) characters.
\end{example} 

The next example shows that the character group with differential condition $W$ is a natural generalization of the group of locally $L$-analytic characters studied by Schneider-Teitelbaum \cite{ST}.
\begin{example}\label{exmpl:DistrL}
Let $\Q_p\subset  L\subset K$ with $L$ finite over $\Qp$. The valuation ring $\grT:=\sO_L$ is a free $\Zp$-module of finite rank with Lie algebra $\frt=L$. We consider the one-dimensional subspace
\[
	W:=\Hom_{L}(L,L)\subseteq \Hom_{\Qp}(L,L)=\Hom_{\Zp}(\grT,L)
\]
of $L$-linear maps from $L$ to $L$. In this case, we have
\begin{align*}
	C_W^{\an}(\grT,K)&= C^{\Lan}(\sO_L,K)\\
	\CharVar{(\sO_L)}_{W}(K)&= \CharVar{(\sO_L)}^{\Lan}(K),
\end{align*}
where  $C^{\Lan}(\sO_L,K)$ denotes the space of all locally $L$-analytic functions on $\sO_L$ and $\CharVar{(\sO_L)}^{\Lan}(K)$ is the group of locally $L$-analytic characters studied in \cite{ST}.
\end{example}

The previous example can be generalized as follows. The notion of $\Sigma$-analytic functions, plays an important role in the interpolation of $p$-adic $L$-functions for algebraic Hecke  characters.
\begin{definition}\label{def:SigmaAnalytic}
Let $L$ be a finite extension of $\Qp$ and write $L^{\mathrm{Gal}}$ for the Galois closure of $L$ in $\Cp$. For any complete field extension $K$ of $L^{\Gal}$ and any subset $\Sigma\subseteq \Hom_{\Qp\text{-alg}}(L,K)$, we say that $f\in C^{\an}(\sO_L,K)$ is \emph{$\Sigma$-analytic} if for every $x\in \sO_L$ there exists an open neighbourhood $U$ of $x$ such that $f|_U$ can be written as a convergent series
\[
	f|_U(z)=\sum_{\underline{n}\in \NN_0^\Sigma} a_{\underline{n}} z^{\underline{n}}, \quad a_{\underline{n}}\in K
\]
where $z^{\underline{n}}:=\prod_{\sigma\in \Sigma} \sigma(z)^{n_\sigma}$. We write $C^{\Sigma\text{-an}}(\sO_L,K)$ for the space of locally $\Sigma$-analytic functions on $\sO_L$.
\end{definition}

Locally $\Sigma$-analytic functions form a natural generalization of locally $L$-analytic functions on $\sO_L$. Indeed, for $\Sigma:=\{L\subseteq K\}\subseteq  \Hom_{\Qp\text{-alg}}(L,K)$ given by the singleton consisting of our fixed inclusion $L\subseteq K$, we recover just the space of locally $L$-analytic functions with values in $K$. The space of locally $\Sigma$-analytic functions can also be described in terms of a differential condition:
\begin{example}\label{exmpl:DistrWSigma}
Let  $L$ be a finite extension of $\Qp$ and $K$ a complete field containing the Galois closure $L^{\mathrm{Gal}}$.
We have 
\begin{multline*}
	\Hom_{\Zp}(\sO_L,K)=\Hom_{K}(L\otimes_{\Qp}K,K)\\
	=\Hom_{L}\left(\bigoplus_{\sigma\in \Hom_{\Qp\text{-alg}}(L,K)} K,K\right)=\prod_{\sigma\in \Hom_{\Qp\text{-alg}}(L,K)} \Hom_{K}(K,K).
\end{multline*}
So, any subset $\Sigma\subseteq \Hom_{\Qp\text{-alg}}(L,K)$ gives rise to a $K$-linear subspace $W(\Sigma)\subseteq \Hom_{\Zp}(\sO_L,K)$ corresponding to $\prod_{\sigma\in \Sigma} \Hom_{K}(K,K)$ under the above isomorphism. It is readily checked that
\[
	C^{\an}_{W(\Sigma)}(\sO_L,K)=C^{\Sigma\text{-an}}(\sO_L,K).
\]
\end{example}
We will see later that the $p$-adic Fourier theory of locally $\Sigma$-analytic functions with values in $\Cp$ is closely related to $p$-divisible groups over $\OCp$ with CM by $\sO_L$.

\section{Character varieties with differential conditions}
In this section we show that the character group $\CharVar{\grT}_W(K)$ is the set of closed points in $K$ of a rigid analytic variety over $L$.
As a preparation, we remark that for characters it suffices to check the differential condition at $0\in \grT$:
\begin{lemma}\label{lem:char-diagram} The group $\CharVar{\grT}_W(K)$ sits in a cartesian diagram 
\begin{equation*}
    \xymatrix{\CharVar{\grT}_W(K)\ar[r]\ar[d]&\CharVar{\grT}(K)\ar[d]^{(d-)|_0}\\
W\ar[r]& \Hom_{\Zp}(\grT,K),}
\end{equation*}
where the right vertical arrow is the map $\chi \mapsto (d\chi)|_0$, which assigns to $\chi$ the value of $d\chi$ at $0\in \grT$.
\end{lemma}
\begin{proof}
%For any $\mathfrak{x}\in \grT$, we have the contraction
%\[
%	\iota_{\mathfrak{x}}\colon C^{\an}(\grT,K)\otimes_K \Hom_{\Zp}(\grT,K)\to C^{\an}(\grT,K), \quad f\otimes\beta\mapsto \beta(\mathfrak{x})\cdot f
%\]
For $g\in \grT$, we have
\[
	(d\chi)|_g=\chi(g) \cdot (d\chi)|_0,
\]
hence $d\chi \in C^{\an}(\grT,K)\otimes_K W$ if and only if $(d\chi)|_0\in W$. This shows
\[
	\CharVar{\grT}_W(K)=\{ \chi \in \CharVar{\grT}(K) \mid (d\chi)|_0\in W\},
\]
and the Lemma follows.
\end{proof}
The group $\CharVar{\grT}(K)$ of locally $\Qp$-analytic characters is the group of $K$-valued points of a rigid analytic variety. More precisely, we have
\[
	\Hom_{\Zp}(\grT,\Zp)\otimes_{\Zp}B_1(K) \xrightarrow{\sim} \CharVar{\grT}(K),\quad \beta\otimes z \mapsto (g\mapsto \chi_{\beta\otimes z}(g):=\exp(\beta(g)\cdot \log z)),
\]
where $B_1(K):=\{z\in K:|z-1|<1\}$ are the $K$-valued points of the rigid analytic generic fiber $(\Gmf)_\eta^{\rig}$ of the formal multiplicative group. The following theorem shows that this result generalizes to character groups with differential condition:
\begin{theorem}\label{thm:AdicCharVariety}
Let $\grT$ be a free $\Zp$-module of finite rank and $W\subset \Hom_{\Zp}(\grT, L)$ be a $L$-subvector space. For any complete field $K$ with $L\subseteq K\subseteq \Cp$, the character group $\CharVar{\grT}_W(K)$ is the set of $K$-valued points of a rigid analytic space $\CharVar{\grT}_W$ over $L$. The rigid analytic space $\CharVar{\grT}_W$ fits into a cartesian diagram
\begin{equation}\label{eq:FiberProd}
    \xymatrix{\CharVar{\grT}_W\ar[r]\ar[d]&	\Hom_{\Zp}(\grT,\Zp)\otimes_{\Zp} (\Gmf)_\eta^{\rig}\ar[d]^{\log}\\
W\otimes_L \bbA^1 \ar[r]& \Hom_{\Zp}(\grT,L)\otimes_{L}\bbA^1=\Hom_{\Zp}(\grT,\Zp)\otimes_{\Zp}\bbA^1,}
\end{equation}
where $\bbA^1$ denotes the rigid analytic affine space over $L$ and $(\Gmf)_\eta^{\rig}$ denotes the rigid analytic generic fiber of the formal multiplicative group $\Gmf$ over $L$.
\end{theorem}
\begin{proof}
Let us define $\CharVar{\grT}_W$ as the adic space defined by the fiber product in  \eqref{eq:FiberProd}. We have to show that the group of $K$-valued points of  $\CharVar{\grT}_W$ is exactly the character group with differential condition $W$.

 Note that the $K$-valued points of $\Hom_{\Zp}(\grT,\Zp)\otimes_{\Zp} (\Gmf)_\eta^{\rig}$ are exactly the locally $\Qp$-analytic characters of $\grT$. More precisely, the group of $K$-valued points of $(\Gmf)_\eta^{\rig}$ is given by
 \[
 	(\Gmf)_\eta^{\rig}(K)=\{z\in K:|z-1|<1\},
 \]
 and we get an isomorphism
 \begin{equation}\label{eq:QpCharIso}
 	\Hom_{\Zp}(\grT,\Zp)\otimes_{\Zp} (\Gmf)_\eta^{\rig}(K)\xrightarrow{\sim} \CharVar{\grT}(K), \quad \beta\otimes z \mapsto (g\mapsto \chi_{\beta\otimes z}(g)),
 \end{equation}
 with $\chi_{\beta\otimes z}(g):=z^{\beta(g)}=\exp(\beta(g)\cdot \log z)$.
 The formula 
 \[
 	(d\chi_{\beta\otimes z})|_0=\beta \cdot \log z \in \Hom_{\Zp}(\grT,K)
 \]
 shows that the isomorphism \eqref{eq:QpCharIso} fits into a commutative diagram
\begin{equation*}
    \xymatrix{\Hom_{\Zp}(\grT,\Zp)\otimes_{\Zp} (\Gmf)_\eta^{\rig}(K)\ar[r]^-{\sim}\ar[d]^{\id\otimes \log}&	\CharVar{\grT}(K)\ar[dl]^-{(d-)|_0}\\
\Hom_{\Zp}(\grT,\Zp)\otimes_{\Zp} \bbA^1 .}
\end{equation*}
  This shows that the $K$-valued points of the lower-left part of the diagram \eqref{eq:FiberProd} is isomorphic to:
\begin{equation*}
    \xymatrix{&	\CharVar{\grT}(K) \ar[d]^{(d-)|_0}\\
W \ar[r]& \Hom_{\Zp}(\grT,K),}
\end{equation*}
and the Theorem follows from Lemma \ref{lem:char-diagram}.
\end{proof}
\begin{definition} For a finite free $\Zp$-module $\grT$ and $W\subseteq\Hom_{\Zp}(\grT,L)$, we call the rigid analytic space $\CharVar{\grT}_W/L$ the
\emph{character variety with differential condition $W$}. We denote by $\CatCharVar_L$ the full subcategory of the category of rigid analytic group varieties over $L$ whose objects are the character varieties $\CharVar{\grT}_W$ for $T$ a finite free $\Zp$-module and $W\subseteq\Hom_{\Zp}(\grT,L)$.
\end{definition}

\begin{remark}
Note that we have the following functoriality: Let $T_1$ and $T_2$ be finite free $\Zp$-modules together with $L$-subvector spaces  $W_1\subseteq \Hom_{\Zp}(T_1,L)$ and $W_2\subseteq \Hom_{\Zp}(T_2,L)$. Any morphism $\varphi\colon T_2\to T_1$ such that $\varphi^*\colon \Hom_{\Zp}(T_1,L)\to  \Hom_{\Zp}(T_2,L)$ satisfies $\varphi^*(W_1)\subseteq W_2$ induces a morphism in $\CatCharVar_L$:
\[
	\CharVar{(\grT_1)}_{W_1}\to \CharVar{(\grT_2)}_{W_2}.
\]
\end{remark}

\section{The Fourier transform} Let $\grT$ be a free $\Zp$-module of finite rank as before and consider complete fields $\Qp\subseteq L\subseteq K\subseteq \Cp$. We will keep the notation $\frt:=\grT\otimes_{\Zp}\Qp$ for the Lie algebra of $\grT$ as a $\Qp$-analytic group. Recall that the algebra of $K$-valued distributions $D(\grT,K)$ is the dual of the Banach space $C^{\an}(\grT,K)$ equipped with the strong dual topology. This is a commutative Fr\'echet $K$-algebra, where multiplication is given by the convolution product $\ast$.  We denote the evaluation of $\mu\in D(\grT,K)$ on functions $h\in C^{\an}(\grT,K)$
by
\begin{equation*}
    \int_\grT h\,d\mu:=\mu(h).
\end{equation*}
The \emph{Fourier transform} of $\lambda \in D(\grT,K)$ is the map
\[
	F_\lambda\colon \CharVar{\grT}(\Cp)\to \Cp,\quad \chi\mapsto \lambda(\chi).
\]
It is not difficult to see that $F_{\lambda\ast \mu}=F_\lambda \cdot F_\mu$ for all $\lambda,\mu \in D(\grT,K)$ (see \cite[Proposition 1.4]{ST}). The main result of $p$-adic Fourier theory for $\Qp$-analytic function is:
\begin{theorem}[Amice]\label{thm:Amice}
The Fourier transform induces an isomorphism of Fr\'echet algebras
\[	
	F\colon D(\grT,K)\xrightarrow{\sim} \cO(\CharVar{\grT}/K).
\]
\end{theorem}
\begin{proof}
See \cite[Theorem 1.3]{Amice78} or \cite[Theorem 2.2]{ST}.
\end{proof}
 We denote by $\mu_f$ the distribution, whose Fourier transform is the function $f\in \sO(\CharVar{\grT})$, so that by definition
\begin{equation}\label{eq:1-int-formula}
    \int_\grT\chi\,d\mu_f=f(\chi).
\end{equation}
The goal of this subsection is to generalize the theorem by Amice to $\Qp$-analytic functions with a differential condition. From now on, let $W\subseteq \Hom_{\Zp}(\grT,L)$ be an $L$-subvector space.
\begin{definition}
	The Fr\'echet algebra of $K$-valued distributions with differential condition $W$ is the continuous dual $D_W(\grT,K):=C^{\an}_W(\grT,K)'$ of the $K$-Banach algebra $C^{\an}_W(\grT,K)$ equipped with the strong dual topology. The product is given by the convolution of distributions.
\end{definition}
The following is a generalization of \cite[Proposition 1.4]{ST} to distributions with differential condition:
\begin{proposition}
For $\lambda \in D_W(\grT,K)$, we have $\lambda=0$ if and only if $F_\lambda|_{\CharVar{\grT}_W(K)}=0$.
\end{proposition}
\begin{proof}
We follow the proof of \cite[Proposition 1.4]{ST}. For any open subset $U\subseteq \grT$ and any function $f$ on $\grT$, let us write $f|U$ for the function which is the extension of $f|_U$ by zero. Any element $h\in \Sym^\bullet_K (W\otimes_LK)$ can be seen as a locally analytic function $h \colon \grT\to K$. In fact, by the definition of $C_W^{\an}(\grT,K)$, the set of functions $h|U$ where $U$ runs through open subsets of $\grT$ and $h$ through $\Sym^\bullet_K( W\otimes_LK)$ is dense in $C_W^{\an}(\grT,K)$.
Suppose now that $F_\lambda|_{\CharVar{\grT}_W(K)}=0$, i.e. that $\lambda(\chi)=0$ for all $\chi\in \CharVar{\grT}_W(K)$. Using the character theory of finite abelian groups, we conclude
\[
	\lambda(\chi|U)=0, \quad \text{ for all open $U\subseteq \grT$ and all } \chi \in  \CharVar{\grT}_W(K).
\]
For any $\beta\in W$ and $z\in K^{\circ\circ}$ of sufficiently small norm, we consider the character $\chi_{\beta\otimes z}\in \CharVar{\grT}_W(K)$, and obtain by continuity for $U\subseteq \grT$ open:
\[
	0=\lambda(\chi_{\beta\otimes z}|U)=\sum_{n\geq 0}\frac{z^n}{n!}\lambda(\beta^n|U).
\]
Since the right hand side is analytic in $z$ in a neighbourhood of $0\in K^{\circ\circ}$, we deduce $\lambda(\beta^n|U)=0$ for all $\beta\in W$, $n\geq 0$ and $U\subseteq \grT$ open. Now note that over a field   $k$ of characteristic zero and a finite dimensional $k$-vector space $V$ any $\Sym^\bullet_k V$ is spanned over $k$ by the set $\beta^n$ with $\beta \in V$ and $n\geq 0$, see for example \cite[Theorem 1]{Schinzel}. Applying this to $k=K$ and $V=W\otimes_LK$, we deduce that $\lambda(h|U)=0$ for any $h\in \Sym^\bullet_K( W\otimes_LK)$ and any $U\subseteq \grT$ open. Since these functions are dense in $C_W^{\an}(\grT,K)$, we deduce $\lambda=0$.
\end{proof}
An immediate consequence of this Proposition is the following Corollary:
\begin{corollary}\label{cor:Ideal-IW}
	Let $W\subseteq \Hom_{\Zp}(\grT,L)$ be a $L$-subvector space. We define
	\[
		I_W:=\ker( D(\grT,K)\to D_W(\grT,K) ),
	\]
	where the surjection is induced by the inclusion $C^{\an}_W(\grT,K)\subseteq C^{\an}(\grT,K)$. Then
	\[
		I_W=\{\lambda\in D(\grT,K): F_\lambda|_{\CharVar{\grT}_W(K)}=0\}.
	\]
\end{corollary}

As we have identified the Lie algebra $\frt$ of $\grT$ with $\grT\otimes_{\Zp}\Qp$, we have an action of $\frt=\grT\otimes_{\Zp}\Qp$ on $C^{\an}(\grT,K)$  in the usual way: for $X\in \frt$ we let
\begin{equation*}
    X:C^{\an}(\grT,K)\to C^{\an}(\grT,K)
\end{equation*}
be the composition of 
\begin{equation*}
    	d\colon C^{\an}(\grT,K)\to C^{\an}(\grT,K)\otimes_K \Hom_{\Zp}(\grT,K)
\end{equation*}
with the evaluation map 
\begin{equation*}
    \Hom_{\Zp}(\grT,K)\times \frt\to K. 
\end{equation*}
%The Lie algebra of $\CharVar{\grT}$ identifies with $\Hom_{\Z_p}(\grT,K)$ 
The resulting action of $X\in \frt$ on $f\in C^{\an}(\grT,K)$ is denoted by $X.f$. This action allows us to define a map
\[
	\iota\colon \frt \to D(\sG,K),\quad X\mapsto (f\mapsto (X.f)(0)).
\]
\begin{lemma}\label{lem:eq-for-charW}
	Let $W\subseteq \Hom_{\Zp}(\grT,L)=\Hom_{\Qp}(\frt,L)$ be a $L$-subvector space, and consider the dual of $W$ in $\frt$:
	 \[
	 W^\perp:=\{ x\in \frt: w(x)=0 \text{ for all } w\in W  \}\subseteq \frt
	 \] 
	 We have
	 \[
		\CharVar{\grT}_W(K)=\{\chi\in \CharVar{\grT}(K): F_{\iota(X)}(\chi)=0 \text{ for all } X\in W^\perp\}.
	\]
\end{lemma}
\begin{proof}
This follows from the following equivalences for a character $\chi\in \CharVar{\grT}$:
\begin{align*}
	&F_{\iota(X)}(\chi)=0 \quad \forall X\in W^\perp\\
	\Leftrightarrow& X.\chi=0, \quad \forall X\in W^\perp\\
	\Leftrightarrow& (d\chi)|_0\in W \\
	\Leftrightarrow& \chi \in \CharVar{\grT}_W(K),
\end{align*}
where the last equivalence follows from Lemma \ref{lem:char-diagram}.
\end{proof}

\begin{remark}
We already know from Theorem \ref{thm:AdicCharVariety} that $\CharVar{\grT}_W$ is a closed reduced rigid analytic subspace of the character variety $\CharVar{\grT}$. Let us choose an $L$-basis $X_1,\dots,X_r$ of $W^\perp$.
 Lemma \ref{lem:eq-for-charW} tells us more explicitly that $\CharVar{\grT}_W$ is cut out by the equations 
 \[
 F_{\iota(X_1)}=0,\dots, F_{\iota(X_r)}=0.
 \]
\end{remark}

We are now ready to prove a generalization of the theorem of Amice under a differential condition $W$:
\begin{theorem}\label{thm:pAdicFourierMainTheoremW}
	The Fourier transform induces an isomorphism of Fr\'echet algebras
	\[
		F_W\colon D_W(\grT,K) \xrightarrow{\sim} \cO(\CharVar{\grT}_W/K), \quad \lambda\mapsto F_\lambda|_{\CharVar{\grT}_W},
	\]
	which fits into a commutative diagram
	\[
		\xymatrix{
			F\colon D(\grT,K) \ar[r]^{\cong}\ar[d] & \cO(\CharVar{\grT}/K)\ar[d] \\
			F_W\colon D_W(\grT,K) \ar[r]^{\cong} & \cO(\CharVar{\grT}_W/K).
		}
	\]
	Here, the left vertical map is induced by the inclusion $C_W^{\an}(\grT,K)\subseteq C^{\an}(\grT,K)$ and the right vertical map is induced by the closed immersion $\CharVar{\grT}_W\subseteq \CharVar{\grT}$.
\end{theorem}
\begin{proof}
By \cite[9.5.2, Cor. 6]{BGR} and \cite[9.5.3, Prop. 4]{BGR}, the coherent sheaf $\cI_W\subseteq \cO_{\CharVar{\grT}}$ corresponding to the closed reduced subvariety $\CharVar{\grT}_W\subseteq \CharVar{\grT}$ consists of all local sections vanishing on the subvariety $\CharVar{\grT}_W\subseteq \CharVar{\grT}$. By Theorem \ref{thm:Amice} (the Theorem of Amice) and Corollary \ref{cor:Ideal-IW}, we have $\cI_W(\CharVar{\grT})=F(I_W)$. Since $\CharVar{\grT}$ is a Stein space, the global section functor is exact on coherent sheaves and we get
\[
	D_W(\grT,K)=D(\grT,K)/I_W\xrightarrow{\sim}\cO(\CharVar{\grT}/K)/\cI(\CharVar{\grT})\cong \cO(\CharVar{\grT}_W).
\]
\end{proof}

Of course, the case $W=\Hom_{\Zp}(\grT,L)$ gives just back the theorem of Amice:
\begin{corollary}[Amice]
	Let $T$ be a finite free $\Zp$-module and set $W:=\Hom_{\Zp}(T,L)$. In this case, the isomorphism $F_W$ coincides with the isomorphism of Amice:
\[
	 D(T,K)\xrightarrow{\sim} \cO(\CharVar{T}/K).
\]
between locally $L$-analytic distributions and the corresponding character variety.
\end{corollary}
In the case $W=\{0\}$, we recover the well-known result that the algebra of locally-constant $K$-valued distributions on $T$ is just the completed group ring:
\begin{corollary}
	Let $T$ be a finite free $\Zp$-module and set $W:=\{0\}$. In this case, the isomorphism $F_W$ gives:
\[
	 D^{\text{lc}}(T,K)\xrightarrow{\sim} K\llbracket T\rrbracket,
\]
where $K\llbracket T\rrbracket$ is the completed group ring of $T$.
\end{corollary}
\begin{proof}
Every locally constant character $\chi\in \CharVar{T}^{\text{lc}}(K)=\CharVar{T}_{\{0\}}(K)$ factors through $T/p^nT$ for some $n$. This shows that
\[
	\CharVar{T}^{\text{lc}}=\varinjlim_n \Sp(K[T/p^n T]),
\]
where $K[T/p^n T]$ is just the group ring of $T/p^nT$ over $K$, and the corollary follows.
\end{proof}

The main theorem of locally $L$-analytic $p$-adic Fourier theory is obtained as the special case with $\grT=\sO_L$ and $W=\Hom_{L}(L,K)\subseteq \Hom_{\Qp}(L,K)$:

\begin{corollary}[{\cite[Theorem 2.3]{ST}}]
Consider a finite field extension $L$ of $\Qp$, and set $T:=\sO_L$ and $W:=\Hom_{L}(L,L)\subseteq \Hom_{\Qp}(L,L)$. In this case, the isomorphism $F_W$ coincides with the isomorphism in \cite[Theorem 2.3]{ST}:
\[
	 D^{\Lan}(\sO_L,K)\xrightarrow{\sim} \cO(\CharVar{(\sO_L)}^{\Lan}/K).
\]
between locally $L$-analytic distributions and the corresponding character variety.
\end{corollary}
\begin{proof}
This follows from Theorem \ref{thm:pAdicFourierMainTheoremW} and Example \ref{exmpl:DistrL}.
\end{proof}

Let us come back to the general setup $\Qp\subseteq L\subseteq K\subseteq \Cp$, $T$ a finite free $\Zp$-module and $W\subseteq \Hom_{\Zp}(T,L)$. For $f\in \cO(\CharVar{\grT}_W/K)$, let us write $\mu_f\in D_W(\grT,K)$ for the corresponding distribution with $F_W(\mu_f)=f$. By Theorem \ref{thm:AdicCharVariety}, the Lie algebra of the rigid analytic variety $\CharVar{\grT}_W/K$ is given by $W_K:=W\otimes_L K$, and $\beta \in W_K$ acts on $\cO(\CharVar{\grT}_W/K)$ as a linear invariant differential operator:
\[
	\beta\colon \cO(\CharVar{\grT}_W/K) \xrightarrow{d} \cO(\CharVar{\grT}_W/K)\otimes_K \Hom_K(W_K,K) \xrightarrow{\text{eval}_\beta} \cO(\CharVar{\grT}_W/K),\quad f\mapsto \beta.f.
\]
On the other hand, $\beta\in W_K$ can be seen as an element $C^{\an}_W(\grT,K)$. These two different interpretations of $W_K$ allow us to view an element $X\in \Sym_K^\bullet W_K$ in two different ways:
\begin{enumerate}
\item as an invariant differential operator $\cO(\CharVar{\grT}_W/K) \to \cO(\CharVar{\grT}_W/K), f\mapsto X.f$,
\item as an element $X\in C^{\an}_W(\grT,K)$.
\end{enumerate}
The second interpretation induces for each $X\in \Sym_K^\bullet W_K$ a $K$-linear map
\[
	\mult_X^*\colon D_W(\grT,K)\to D_W(\grT,K),
\]
which is dual to the multiplication map 
\[
	\mult_X\colon C^{\an}_W(\grT,K)\to C^{\an}_W(\grT,K), f\mapsto X\cdot f.
\]
These two interpretations are related by the following integration formula:
\begin{proposition}
	For $f\in \cO(\CharVar{\grT}_W/K)$ and $X\in \Sym_K^\bullet W_K$, we have:
	\[
		\mult_X^*\mu_f=\mu_{X.f},
	\]
	or stated differently: For every $\chi \in \CharVar{\grT}_W(K)$, we have:
	\begin{equation*}
	\int_{\grT} \chi\cdot X\,d\mu_f=(X.f)(\chi).
	\end{equation*}
\end{proposition}
\begin{proof}
In the case $W=\Hom_\Zp(\grT,L)$, the Fourier transform $F_W$ is just the classical Amice transform $F$ and the statement is well-known, see e.g. \cite[Lemma 4.6 (8)]{ST}. For general $W\subseteq \Hom_{\Zp}(\grT,K)$, the statement follows from Theorem \ref{thm:pAdicFourierMainTheoremW} and the commutativity of
\[
		\xymatrix{
			I_W 		\ar[r]^{\mult_X^*}\ar[d]^{F} & I_W \ar[d]^{F} \\
		 	F(I_W) 	\ar[r]^{X.} & F(I_W) .
		}
\]
\end{proof}
For later reference, let us observe the following tautological integration formula:
\begin{lemma}
	Any finite character $\chi_{\text{fin}}\colon \grT\to K^\times$ is contained in $\CharVar{\grT}_W(K)$ and we have
	\[
		\mult_{\chi_{\text{fin}}}^*\mu_f=f((-)\cdot\chi_{\text{fin}}).
	\]
\end{lemma}
\begin{proof}
A finite character is locally constant and hence in $\CharVar{\grT}_W(K)$ for any differential condition $W$. The integration formula follows immediately from the definitions: For any $\varphi\in \CharVar{\grT}_W(K)$, we have:
\[
	(\mult_{\chi_{\text{fin}}}^*\mu_f)(\varphi)=\mu_f(\varphi\cdot \chi_{\text{fin}})=f(\varphi\cdot \chi_{\text{fin}}).
\]
\end{proof}

%There is another way to look at the Fourier transform, following Katz \cite{Katz-formal-groups}. Let $g\in \grT$, then evaluation at $g$ gives a locally analytic function
%\begin{align*}
%    \ev_g:\CharVar{\grT}\to K&& \chi\mapsto \chi(g).
%\end{align*}
%Let $\Diff(\CharVar{\grT})$ be the algebra of invariant differential operators of $\CharVar{\grT}$. Then for any $D\in \Diff(\CharVar{\grT})$ one gets a function 
%$\langle D,\chi\rangle\in C^{\an}(\grT,K)$ (\TODO: check that is true) by setting
%\begin{equation*}
%    \langle D,\chi\rangle(g):=(D\chi)(g).
%\end{equation*}
%This gives a map
%\begin{equation*}
%    \Diff(\CharVar{\grT})\to C^{\an}(\grT,K).
%\end{equation*}
%On the other hand, one has a pairing
%\begin{align}
%    \Diff(\CharVar{\grT})\times\sO(\CharVar{\grT})\to K&&(D, f)\mapsto Df(1)
%\end{align}
%where $1$ is the trivial character. 

\section{$p$-divisible groups over $\sO_{\C_p}$} Here we review the classification of $p$-divisible groups over $\sO_{\C_p}$ by Scholze--Weinstein. Let us denote by $\pDiv_{\OCp}$ the category of $p$-divisible groups over $\OCp$. The other category appearing in the classification result of Scholze and Weinstein is the following category of pairs:
\begin{definition}
We write $\HTpairs$ for the following category:
\begin{itemize}
\item \textbf{Objects} given by pairs $(W,T)$ with $T$ a free $\Zp$-module of finite rank and $W\subseteq T\otimes_{\Zp}\Cp(-1)$.
\item \textbf{Morphisms}  $(W_1,T_1)\to (W_2,T_2)$  are $\Zp$-linear maps $\varphi \colon T_1\to T_2$ such that the induced map $T_1\otimes \Cp(-1)\to T_2\otimes \Cp(-1)$ maps $W_1$ into $W_2$.
\end{itemize}
\end{definition}
For a $p$-divisible group $G$ over $\OCp$, the injective map in the Hodge--Tate exact sequence for the dual $p$-divisible group $G^\vee$ gives an inclusion
\[
	\Lie G\otimes_{\OCp}\Cp \subseteq T_p(G)\otimes_{\Zp}\Cp(-1).
\] 
Hence, we obtain a functor 
\[
	HT\colon\pDiv_{\OCp}\to \HTpairs,\quad G\mapsto (\Lie G\otimes \Cp, T_p(G)).
\]
Scholze and Weinstein have shown that the functor $HT$ is an equivalence of categories (see \cite[Theorem 5.2.1]{ScholzeWeinstein}). By Cartier duality, we have for any $G\in \mathrm{ob}(\pDiv_{\OCp})$ a canonical isomorphism
\[
	\Hom_{\Zp}(T_p(G^\vee),\Cp)\cong T_p(G)\otimes_{\Zp}\Cp(-1).
\]
For our purposes, it is more convenient to express the classification of $p$-divisible groups in terms of the Tate module $T_p(G^\vee)$  instead of $T_p(G)$. Therefore, we define:
\begin{definition}
We write $\HTDpairs$ for the following category:
\begin{itemize}
\item \textbf{Objects} given by pairs $(W,T)$ with $T$ a free $\Zp$-module of finite rank and $W\subseteq \Hom_{\Zp}(T,\Cp)$.
\item \textbf{Morphisms}  $(W_1,T_1)\to (W_2,T_2)$  are $\Zp$-linear maps $\varphi \colon T_2\to T_1$ such that the induced map $\Hom_{\Zp}(T_1,\Cp)\to \Hom_{\Zp}(T_2,\Cp)$ maps $W_1$ into $W_2$.
\end{itemize}
\end{definition}
In terms of $T_p(G^\vee)$, the injective part of the Hodge--Tate sequence reads
\[
	\Lie G\otimes_{\OCp}\Cp \subseteq \Hom_{\Zp}(T_p(G^\vee),\Cp),
\] 
and we obtain a functor
\[
	HT^\vee\colon \pDiv_{\OCp}\to \HTDpairs, \quad G\mapsto (\Lie G\otimes_{\OCp}\Cp ,T_p(G^\vee)).
\]
Of course, there is an equivalence of categories
\[
	\HTpairs\xrightarrow{\sim} \HTDpairs,\quad (W,T)\mapsto (W,\Hom(T,\Zp(1))),
\]
fitting into the commutative diagram:
\begin{equation}\label{eq:commDiagHT}
	\xymatrix{
			\HTpairs\ar[rr]^{\sim} & & \HTDpairs\\
			& \pDiv. \ar[lu]^{HT}\ar[ru]_{HT^\vee} &
	}
\end{equation}
In terms of the functor $HT^\vee$, the classification result for $p$-divisible groups over $\OCp$ of
Scholze and Weinstein can be formulated as follows:

\begin{theorem}[Scholze--Weinstein]\label{thm:ScholzeWeinstein}
We have an equivalence of categories:
\begin{align*}
		HT^\vee \colon \pDiv_{\OCp}\xrightarrow{\sim} \HTDpairs, \quad G\mapsto (\Lie G\otimes \C_p, T_p(G^\vee)).
\end{align*}
\end{theorem}
\begin{proof}
This follows from \cite[Theorem 5.2.1]{ScholzeWeinstein} using \eqref{eq:commDiagHT}.
\end{proof}
We refer to \cite[Theorem 5.2.1]{ScholzeWeinstein} for a full proof, but let us briefly recall the following explicit description of a quasi-inverse
\[
	\HTDpairs\xrightarrow{\sim} \pDiv_{\OCp}, \quad (W,T)\mapsto G_{(W,T)}.
\]
Given $(W,T)\in \mathrm{ob}(\HTDpairs)$, the logarithm of the formal multiplicative group $\Gmf$ induces a map
\begin{equation}\label{eq:WtoLieG}
	\log\colon \Hom_{\Zp}(T,\Zp)\otimes_\Zp (\Gmf)^{\ad}_\eta \to \Hom_{\Zp}(T,\Cp)\otimes_{\Cp} \bbA^1,
\end{equation}
where $(\Gmf)^{\ad}_\eta$ is the (adic) generic fiber of the formal multiplicative group and $\bbA^1$ denotes the adic affine line over $\Cp$.
The following pullback diagram defines a $p$-divisible rigid-analytic group in the sense of \cite{Fargues19},
\begin{equation}\label{eq:ScholzeWeinsteinCartDiag}
	\xymatrix{
		(G_{(W,T)})_\eta^{\ad} \ar[r]\ar[d] &  \Hom_{\Zp}(T,\Zp)\otimes_\Zp (\Gmf)^{\ad}_\eta \ar[d]^-{\log}\\
		W\otimes_{\Cp} \mathbb{A}^1 \ar[r] & \Hom_{\Zp}(T,\Cp)\otimes_{\Cp} \bbA^1,
	}
\end{equation}
where the lower horizontal map is given by the inclusion $W\subseteq \Hom_{\Zp}(T,\Cp)$. In the proof of \cite[Theorem 5.2.1]{ScholzeWeinstein} it is finally shown that $(G_{(W,T)})_\eta^{\ad}$ is indeed the adic generic fiber of a $p$-divisible group $G_{(W,T)}$, which can be explicitly described as
\[
	G_{(W,T)}:=\coprod_{Y\subseteq (G_{(W,T)})_\eta^{\ad}} \Spf H^0(Y,\cO_Y^+).
\]

\section{Uniformization of character varieties with differential conditions by $p$-divisible groups}
For any complete subfield $K\subseteq \Cp$, there is a canonical fully faithful functor from the category $\Rigid_K$ of $K$-rigid analytic varieties to the category of adic spaces
\[
	\Rigid_K\to \Adic, \quad X\mapsto X^{\ad}.
\]
For a $p$-divisible group $G$ over valuation ring $\sO_K$ of $K$, let us write $G^{\rig}_\eta\in \Rigid_K$ for the rigid analytic generic fiber of $G$. Note that the rigid analytic generic fiber functor is compatible with the adic generic fiber functor, i.e. we have an equivalence $(\cdot)^{\ad}\circ (\cdot)_\eta^{\rig}\simeq (\cdot)^{\ad}_\eta \colon \pDiv\to \Adic$.

\begin{theorem}\label{thm:UniformizationOfCharVar}
Let $L$ be a complete subfield of $\Cp$. 
Let $\grT$ be a free $\Zp$-module of finite rank and $W\subseteq \Hom_{\Zp}(\grT,L)$ a $L$-subvector space. We obtain a pair $(W_{\Cp},T)\in \HTDpairs$ with $W_{\Cp}:=W\otimes_L \Cp$. There is an isomorphism of rigid analytic varieties over $\Cp$
\[
	(G_{(W_{\Cp},T)})_\eta^{\rig} \cong \CharVar{\grT}_W/\Cp,
\]
where $ \CharVar{\grT}_W$ is the character variety of $T$ with differential condition $W$, and $(G_{(W_{\Cp},T)})_\eta^{\rig}$ is the rigid generic fiber of the $p$-divisible group $G_{(W_{\Cp},T)}$ associated to the pair $(W_{\Cp},T)$, see \eqref{eq:ScholzeWeinsteinCartDiag}.
\end{theorem}
\begin{proof}
By \eqref{eq:ScholzeWeinsteinCartDiag}, the adic space $(G_{(W_{\Cp},T)})_\eta^{\ad}$ is given by the Cartesian diagram
\begin{equation*}
	\xymatrix{
		(G_{(W_{\Cp},T)})_\eta^{\ad} \ar[r]\ar[d] &  \Hom_{\Zp}(\grT,\Zp)\otimes_\Zp (\Gmf)_\eta^{\ad} \ar[d]^-{\log}\\
		W_{\Cp}\otimes_{\Cp} \mathbb{A}^1 \ar[r] & \Hom_{\Zp}(\grT,\Zp)\otimes_{\Zp} \mathbb{A}^1.
	}
\end{equation*}
By equation \eqref{eq:FiberProd} in Theorem \ref{thm:AdicCharVariety}, the adic space associated to the character variety $\CharVar{\grT}_W/\Cp$ is given by the same pull-back diagram. The claim follows since the functor $(\cdot)^{\ad}$  commutes with fiber products.
\end{proof}

Recall that we defined $\CatCharVar_{\Cp}$ as the essential image of the functor
\[
	\CharVar{(\cdot)}_{(\cdot)}\colon \HTDpairs\to \RigidGr_{\Cp}, \quad (W,T)\mapsto \CharVar{T}_W,
\]
to rigid analytic group varieties $\RigidGr_\Cp$ over $\Cp$.
\begin{corollary}
	The category $\CatCharVar_{\Cp}$ is the essential image of $\pDiv_{\OCp}$ under the generic fiber functor  and we have $\CharVar{(\cdot)}_{(\cdot)}\simeq (\cdot)^{\rig}_\eta\circ HT^\vee $.
\end{corollary}

Let us discuss some examples of $p$-divisible groups over $\OCp$ and their associated character varieties. We start with the following trivial examples:

\begin{corollary}
Let $T$ be a finite free $\Zp$-module. The pair $(0,T)\in \HTDpairs$ corresponds to the \'etale $p$-divisible group $\Hom_{\Zp}(T,\Zp(1))\otimes \underline{\Qp /\Zp}$ and we obtain the following isomorphism of rigid analytic varieties over $\Cp$
\[
	\CharVar{\grT}^{\text{lc}}\cong \Hom_{\Zp}(T,\Zp(1))\otimes (\underline{\Qp /\Zp})_{\eta}^{\rig}.
\]
\end{corollary}
\begin{proof}
We have $HT^\vee(\Hom_{\Zp}(T,\Zp(1))\otimes \underline{\Qp /\Zp})=(0,T)$.
\end{proof}

\begin{corollary}
Let $T$ be a finite free $\Zp$-module. The pair $(\Hom_{\Zp}(T,\Cp),T)\in \HTDpairs$ corresponds to the $p$-divisible group $\Hom_{\Zp}(T,\Zp)\otimes \mu_{p^\infty}$ and we obtain the following isomorphism of rigid analytic varieties over $\Cp$
\[
	\CharVar{\grT}\cong \Hom_{\Zp}(T,\Zp)\otimes (\Gmf)_{\eta}^{\rig}.
\]
\end{corollary}
\begin{proof}
We have $HT^\vee(\Hom_{\Zp}(T,\Zp)\otimes \mu_{p^\infty})=(\Hom_{\Zp}(T,\Cp),T)$.
\end{proof}

As a less trivial example, let us remark that we get the main result of Schneider--Teitelbaum about the uniformization of locally $L$-analytic character varieties from Theorem \ref{thm:UniformizationOfCharVar}:

\begin{corollary}[{\cite[Theorem 3.6]{ST}}]
Let $L$ be a finite extension over $\Qp$. Let $G$ be the $p$-divisible group over $\OCp$ associated to a Lubin--Tate formal group with endomorphisms by $\sO_L$, then $HT^\vee(G)\cong  (\Hom_{L}(L,L),\sO_L)\in \HTDpairs$. In particular, we have
\[
	\CharVar{T}^{\Lan}\cong G^{\rig}_\eta
\] 
\end{corollary}
\begin{proof}
It suffices to show that $HT^\vee(G)\cong (\Hom_{L}(L,L),\sO_L)\in \HTDpairs$. Note that the image of the injective map in the Hodge-Tate sequence
\[
	\Lie G\otimes_{\OCp} \Cp \hookrightarrow T_pG(-1)\otimes_\Zp \Cp\cong \Hom_{\Zp}(T_p(G^\vee), \Cp).
\]
 is exactly the sub-space of $\Hom_{\Zp}(T_pG^\vee, \Cp)$, where the endomorphisms $\sO_L$ of $F_{LT}$ act by the fixed inclusion $\sO_L\subseteq \Cp$. The Tate module $T_p(G^\vee)$ is a free $\sO_L$-module of rank $1$. Hence, after fixing a generator of $T_p(G^\vee)$ as $\sO_L$-module, we obtain the desired isomorphism in $\HTDpairs$:
 \[
 	HT^\vee(G)\cong (\Lie G\otimes_{\OCp} \Cp, T_p(G^\vee))  \cong (\Hom_{L}(L,L),\sO_L).
 \]
\end{proof}

As a last example, we would like to discuss how $p$-divisible groups with CM by a finite extension $L$ of $\Qp$ fit into the picture.
\begin{definition}
Let $L$ be a finite extension of $\Qp$, $K$ a complete field extension of $L$ and $\Sigma\subseteq \Hom_{\Qp\text{-alg}}(L,K)$ be a subset of field embeddings of $L$ to $K$. A $p$-divisible group \emph{with CM by $\sO_L$ of type $\Sigma$} over the  valuation ring $\sO_K$ of $K$ is a pair $(G,i)$ consisting of a $p$-divisible group $G$ over $\sO_K$ and an embedding
\[
	i\colon \sO_L\hookrightarrow \End(G)
\] 
with $[L:\Qp]=\mathrm{ht}(G)$ such that  we have a decomposition 
\[
	\Lie G\otimes K\cong \bigoplus_{\sigma \in \Sigma} (\Lie G\otimes K)(\sigma)
\]
into $1$-dimensional $K$-vector spaces $(\Lie G\otimes K)(\sigma)$ such that $\sO_L$ acts on $(\Lie G\otimes K)(\sigma)$ by the embedding $\sigma \colon L\to K$.
\end{definition}

\begin{corollary}
	Let $L$ be a finite extension of $\Qp$ and $\Sigma\subseteq \Hom_{\Qp\text{-alg}}(L,\Cp)$. Let $G\in \mathrm{ob}(\pDiv_{\OCp})$ be a $p$-divisible group with CM by $\sO_L$ of type $\Sigma$. Then $HT^\vee(G)=(W(\Sigma), \sO_L)$, where
	\[
		W(\Sigma)\subseteq \Hom_{\Zp}(\sO_L,\Cp)
	\]
	 is the subspace defined in Example \ref{exmpl:DistrWSigma}. In particular, we have 
	 \[
	 	D^{\Sigma\text{-an}}(\sO_L,\Cp)\cong \cO(G^{\rig}_\eta/\Cp),
	 \]
	 where $D^{\Sigma\text{-an}}(\sO_L,\Cp):=C^{\Sigma\text{-an}}(\sO_L,\Cp)'$ is the Fr\'echet space given by the dual of the space of locally $\Sigma$-analytic functions defined in Definition \ref{def:SigmaAnalytic}. 
\end{corollary}
\begin{proof}
By the definition of a $p$-divisible group with CM by $\sO_L$ of type $\Sigma$, $T_p(G^\vee)$ is a free $\sO_L$-module of rank $1$ and $\Lie G\otimes \Cp$ decomposes as $\sO_L\otimes_{\Zp}\Cp$-module into $1$-dimensional $\Cp$-vector spaces $(\Lie G\otimes \Cp)(\sigma)$ with $\sO_L$-action through $\sigma\in \Sigma$. Therefore, we have 
\[
	HT^\vee(G)=(W(\Sigma), \sO_L)
\]
and the corollary follows.
\end{proof}

\begin{remark}
The last example of $p$-divisible groups with CM by $\sO_L$ was our main motivation for extending the theory of Schneider--Teitelbaum from Lubin--Tate groups to more general $p$-divisible groups. In an upcoming paper, we will use this theory to construct $p$-adic $L$-functions (in the sense of $\Sigma$-analytic distributions)  interpolating critical Hecke $L$-values of an arbitrary totally imaginary field.
\end{remark}

%% file: character-varieties.bib
@preamble{"\providecommand{\MR}[1]{}"}

@article{Amice64,
     author = {Amice, Y.},
     title = {Interpolation $p$-adique},
     journal = {Bulletin de la Soci\'et\'e Math\'ematique de France},
     pages = {117--180},
     publisher = {Soci\'et\'e math\'ematique de France},
     volume = {92},
     year = {1964},
     doi = {10.24033/bsmf.1606},
     zbl = {0158.30201},
     language = {fr},
     url = {http://www.numdam.org/articles/10.24033/bsmf.1606/}
}

@incollection {Amice78,
	AUTHOR = {Amice, Y.},
	TITLE={Duals},
     BOOKTITLE = {Proceedings of the {C}onference on {$p$}-adic {A}nalysis},
      NOTE = {Held in Nijmegen, January 16--20, 1978,
              Report, 7806},
 PUBLISHER = {Katholieke Universiteit, Mathematisch Instituut, Nijmegen},
      YEAR = {1978},
     PAGES = {ii+224},
   MRCLASS = {12B40},
  MRNUMBER = {522116},
}

@article {Fargues19,
    AUTHOR = {Fargues, Laurent},
     TITLE = {Groupes analytiques rigides {$p$}-divisibles},
   JOURNAL = {Math. Ann.},
  FJOURNAL = {Mathematische Annalen},
    VOLUME = {374},
      YEAR = {2019},
    NUMBER = {1-2},
     PAGES = {723--791},
      ISSN = {0025-5831,1432-1807},
   MRCLASS = {14L05 (14G22)},
  MRNUMBER = {3961325},
MRREVIEWER = {Ilya\ Karzhemanov},
       DOI = {10.1007/s00208-018-1782-9},
       URL = {https://doi.org/10.1007/s00208-018-1782-9},
}

@article{ScholzeWeinstein,
	author = "Scholze, P. and Weinstein, J.",
	doi = "10.4310/CJM.2013.v1.n2.a1",
	fjournal = "Cambridge Journal of Mathematics",
	journal = "Camb. of Math.",
	number = "2",
	pages = "145--237",
	title = "{Moduli of $p$-divisible groups}",
	url = "https://dx.doi.org/10.4310/CJM.2013.v1.n2.a1",
	volume = "1",
	year = "2013"
}

@article {Schinzel,
    AUTHOR = {Schinzel, A.},
     TITLE = {On a decomposition of polynomials in several variables},
   JOURNAL = {J. Th\'eor. Nombres Bordeaux},
  FJOURNAL = {Journal de Th\'eorie des Nombres de Bordeaux},
    VOLUME = {14},
      YEAR = {2002},
    NUMBER = {2},
     PAGES = {647--666},
      ISSN = {1246-7405,2118-8572},
   MRCLASS = {11C08 (11D85 12E05)},
  MRNUMBER = {2040699},
MRREVIEWER = {Mieczys\l aw\ Kula},
       URL = {http://jtnb.cedram.org/item?id=JTNB_2002__14_2_647_0},
}

@article {ST,
    AUTHOR = {Schneider, P. and Teitelbaum, J.},
     TITLE = {{$p$}-adic {F}ourier theory},
   JOURNAL = {Doc. Math.},
  FJOURNAL = {Documenta Mathematica},
    VOLUME = {6},
      YEAR = {2001},
     PAGES = {447--481},
      ISSN = {1431-0635,1431-0643},
   MRCLASS = {11S31 (11G40 14G22 46S10)},
  MRNUMBER = {1871671},
MRREVIEWER = {Takao\ Yamazaki},
}

@book {BGR,
    AUTHOR = {Bosch, S. and G\"untzer, U. and Remmert, R.},
     TITLE = {Non-{A}rchimedean analysis},
    SERIES = {Grundlehren der mathematischen Wissenschaften [Fundamental
              Principles of Mathematical Sciences]},
    VOLUME = {261},
      NOTE = {A systematic approach to rigid analytic geometry},
 PUBLISHER = {Springer-Verlag, Berlin},
      YEAR = {1984},
     PAGES = {xii+436},
      ISBN = {3-540-12546-9},
   MRCLASS = {32K10 (30G05 46P05)},
  MRNUMBER = {746961},
MRREVIEWER = {W.\ Bartenwerfer},
       DOI = {10.1007/978-3-642-52229-1},
       URL = {https://doi.org/10.1007/978-3-642-52229-1},
}

@unpublished{Kings-Sprang,
      title={Eisenstein-{K}ronecker classes, integrality of critical values of {H}ecke {L}-functions and p-adic interpolation}, 
      author={G. Kings and J. Sprang},
      year={},
      eprint={1912.03657},
      archivePrefix={arXiv},
      note={to appear in {A}nnals of {M}athematics},
}
